\documentclass[centertags, reqno]{amsart}                       
\usepackage{graphicx}
\usepackage{url}     
\usepackage{dsfont}

\newtheorem{thm}{Theorem}[section]
\newtheorem*{thm*}{Theorem} 
\newtheorem{lemma}[thm]{Lemma}
\newtheorem{cor}[thm]{Corollary}
\newtheorem{prop}[thm]{Proposition}

\theoremstyle{definition}
\newtheorem{definition}[thm]{Definition}

\newtheorem{ex}[thm]{Example} 

\theoremstyle{remark}
\newtheorem{rem}[thm]{Remark}

\newcommand{\mr}{{\mathbb R}}

\newcommand{\mn}{{\mathbb N}}

\newcommand{\mc}{{\mathbb C}}
\newcommand{\md}{{\mathbb D}}

\renewcommand{\rho}{\varrho}
\newcommand{\eps}{\varepsilon}

\renewcommand{\Re}{\operatorname{Re}}

\newcommand{\dist}{\operatorname{dist}}

\newcommand{\tr}{\operatorname{tr}}

\newcommand{\ran}{\operatorname{Ran}}
\newcommand{\hil}{\mathcal{H}}
\newcommand{\bdd}{\mathcal{B}}
\newcommand{\bi}{\mathcal{I}}

\newcommand{\ff}{\mathcal{F}}

\newcommand{\rank}{\operatorname{Rank}}

\newcommand{\dom}{\operatorname{Dom}}

\newcommand{\ac}{\mathcal{A}}

\newcommand{\pp}{\lceil p \rceil}

 \usepackage{xcolor}

\newcommand{\bn}{V}

\numberwithin{equation}{section}

\begin{document}

\title[Perturbation determinants in Banach spaces]{Perturbation determinants in Banach spaces - With an application to eigenvalue estimates for perturbed operators} 
 
\begin{abstract}
In the first part of this paper we provide a self-contained introduction to (regularized) perturbation determinants for operators in Banach spaces. In the second part, we use these determinants to derive new bounds on the discrete eigenvalues of compactly perturbed operators, broadly extending some recent results by Demuth et al. In addition, we also establish new bounds on the discrete eigenvalues of generators of $C_0$-semigroups.
\end{abstract}

\author[M. Hansmann]{Marcel Hansmann}
\address{Faculty of Mathematics\\ 
Chemnitz University of Technology\\
Chemnitz\\
Germany.}
\email{marcel.hansmann@mathematik.tu-chemnitz.de}

\maketitle


\section{Introduction}
 
The theory of perturbation determinants and regularized perturbation determinants for \emph{Hilbert space} operators is very well-developed and has been successfully applied in a large variety of settings (see, e.g., \cite{b_Gohberg69}, \cite{MR0482328}, \cite{b_Simon05}, \cite{MR1180965}, \cite{MR2201310}, \cite{MR3077277} and references therein). On the other hand, with two recent exceptions \cite{DemHanauska13, MR3296588}, regularized perturbation determinants for operators on general Banach spaces seem to have  been hardly considered at all. It is the aim of this paper to provide some first results for a corresponding theory. Moreover, by applying these results to the study of the distribution of eigenvalues of compactly perturbed operators, we will also provide  good reasons why the development of such a theory is a worthwhile endeavor. 

To avoid misunderstandings: we do not claim that there does not exist an extensive determinant theory for operators on Banach spaces. Quite on the contrary, there exist several approaches to define a (regularized) determinant for operators of the form $I-L$, where $I$ denotes the identity operator and $L$ is element of a suitable subclass of (power-)compact operators. Let us mention the work of K\"onig \cite{MR568991} (see also \cite{MR1840556}), the axiomatic determinant theory developed by Pietsch \cite{MR917067}, and the alternative approach by Gohberg et al. \cite{MR1744872}. It is just the fact that these determinants, as far as we can say, have not been used to define \emph{perturbation} determinants, which we alluded to in the previous paragraph. 

What, then, is a perturbation determinant? A very rough description goes as follows: Given a bounded operator $A$ and a compact operator $K$, both defined on a complex Banach space $X$, a perturbation determinant of $A$ by $K$ (here we use the terminology of \cite{b_Gohberg69}) is a holomorphic function $d=d^{A,K}$ defined on the resolvent set $\rho(A)$ of $A$ which has the property that $d(\lambda)=0$ if and only if $\lambda$ is in the spectrum of $A+K$. Moreover, if $\lambda$ is a discrete eigenvalue of this operator, then its algebraic multiplicity should coincide with the order of $\lambda$ as a zero of $d$. The name \emph{perturbation determinant} stems from the fact that here we consider $A$ as some known (free) operator, which is perturbed by the compact operator $K$.  In particular, the perturbation determinant allows us to transfer the problem of studying the spectrum  of the perturbed operator $A+K$ to the classical problem of analyzing the zero set of a holomorphic function.

As an example, let us take a look at the well-studied case of operators acting on a complex Hilbert space $\hil$ (see \cite[Chapter IV.2]{b_Gohberg69} or \cite{MR0482328}). Here, assuming that the operator $L$ is an element of the $n$th von Neumann-Schatten class ${S}_n(\hil), n \in \mn,$ (which consists of those compact operators on $\hil$ whose singular values are in $l_n(\mn)$), one defines the $n$-regularized determinant\footnote{In case $n=1$ we would not speak of a regularized determinant, but simply of a determinant.} of $I-L$ by
\begin{equation}
  \label{eq:37}
 \operatorname{det}_n(I-L):= \prod_j \left( (1-\lambda_j(L)) \exp\left( \sum_{k=1}^{n - 1} \frac {\lambda_j^k(L)}{k} \right) \right),  
\end{equation}
 where $(\lambda_j(L))$ denotes the sequence of discrete eigenvalues of  $L$ (with the convention that $\lambda_{N+1}(L)=\lambda_{N+2}(L)=\ldots=0$ if $L$ has only $N$ discrete eigenvalues). Note that the above product is convergent since the eigenvalue sequence of an operator in $S_n(\hil)$ is in $l_n(\mn)$. Moreover, we see that $\operatorname{det}_n(I-L)=0$ if and only if $1$ is in the spectrum $\sigma(L)$ of $L$. Now, given a bounded operator $A$ on $\hil$ and $K \in S_n(\hil)$, we can use the fact that $S_n(\hil)$ is an ideal in the algebra $\bdd(\hil)$  of all bounded operators on $\hil$ (and so $K(\lambda-A)^{-1} \in S_n(\hil)$) to define the $n$-regularized perturbation determinant of $A$ by $K$ as 
 \begin{equation}
   \label{eq:38}
 d_n^{A,K}(\lambda):= {\operatorname{det}}_n(I-K(\lambda-A)^{-1}), \quad \lambda \in \rho(A).   
 \end{equation}
Then  $d_n^{A,K}(\lambda)=0$ if and only if $1 \in \sigma(K(\lambda-A)^{-1})$ and it is not difficult to see that this is the case if and only if $\lambda \in \sigma(A+K)$. What is not obvious, but still is true, is the fact that this function is holomorphic and that in case that $\lambda$ is a discrete eigenvalue of $A+K$, its algebraic multiplicity (as an eigenvalue) and its order (as a zero) coincide. The easiest way to prove these results is by  a reduction to the case of finite rank perturbations, using the fact that the finite rank operators are dense in $S_n(\hil)$.

In this paper, abstracting from the Hilbert space case sketched above, we will introduce regularized determinants for operators $I-L$ on a complex Banach space $X$, assuming that $L$ is element of a subclass $\bi$ of compact operators on $X$ which has the following properties:
\begin{itemize}
    \item[-] $\bi$ is a (quasi-)normed subspace of $\bdd(X)$ and the set of finite rank operators $\ff(X)$ is dense in $\bi$,
\item [-] $\bi$ is an ideal in the algebra $\bdd(X)$,
\item[-] there exists $p >0$ such that for every $L\in \bi$ the sequence $(\lambda_j(L))$ of discrete eigenvalues of $L$ is in $l_p$.
\end{itemize}  
For reasons of brevity, in this paper such an ideal $\bi$ will be called an $l_p$-ideal\footnote{Concerning this terminology, see Remark \ref{rem42}.} in $\bdd(X)$ (we refer to Section $4$ for the precise definition). Given the $l_p$-ideal $\bi$, we will first define ${\det_n}(I-L)$ for finite rank operators $L \in \ff(X)$ using formula (\ref{eq:37}) (with $n= \lceil p \rceil$), and then show that this function can be continuously extended to all of $\bi$. Having achieved this, the perturbation determinant of $A$ by $K$ (with $K \in \bi$) can be defined as in (\ref{eq:38}) and we will check that it has all the desired properties. 

The idea to define regularized determinants by reduction to finite rank operators already appears in  Pietsch's monograph \cite{MR917067} and has later been systematically developed by Gohberg et al. in \cite{MR1744872}. Indeed, in \cite{MR1744872} we can find several equivalent conditions describing when the regularized determinant (\ref{eq:37}), initially defined on $\ff(X)$, can be continuously extended to all of $\bi$. One of these conditions, the local Lipschitz-continuity of $L \mapsto \operatorname{det}_n(I-L)$, will be a consequence of our third assumption on $\bi$, so that the theory of \cite{MR1744872} applies. To be precise, in order for their extension theory to work,  the authors of \cite{MR1744872} do only require that $\bi$ is a certain sub-algebra of $\bdd(X)$, but our stronger requirement of it being an ideal seems to be necessary in our subsequent construction of the perturbation determinant. Moreover, in contrast to us, the authors of \cite{MR1744872} do only consider the case where $\bi$ is a normed subspace (whereas we allow for a quasi-normed space), so that in this respect our considerations are slightly more general. 

Our results on regularized determinants are also closely related to the axiomatic approach developed by Pietsch in \cite{MR917067}. Pietsch studies the question of the existence of continuous determinants (and regularized determinants) for operators in quasi-Banach operator ideals (see Section \ref{sec:examples} below) and is also interested in the connection between determinants and traces (a topic we will not speak about at all). In particular, the assumptions we made on $\bi$ will be satisfied whenever $\bi$ is a component of an approximative quasi-Banach operator ideal of eigenvalue type $l_p$ (in the sense of Pietsch), meaning that our results allow the construction of perturbation determinants if the perturbation $K$ is, e.g., a nuclear or a $q$-summing operator, or an element of one of the classical $s$-number ideals (like the approximation- or Weyl-number ideals). We refer to Section \ref{sec:examples} for more such examples.

Our main impetus for the present work came from the recent paper \cite{MR3296588}, which studied the distribution of eigenvalues of compactly perturbed operators on Banach spaces. One of the main results of  \cite{MR3296588} reads as follows: Given a bounded operator $A \in \bdd(X)$ and a compact operator $K$ on $X$, let $n_{A+K}(s)$ denote the number of discrete eigenvalues of $A+K$ in $\{ \lambda \in \mc : |\lambda| > s\}$. Then the following holds:
\begin{equation}
  \label{eq:39}
  n_{A+K}(s) \leq   C_{p} \frac{s}{(s-\|A\|)^{p+1}} \sum_{ j \in \mn} a_j(K)^p, \qquad s> \|A\|,
\end{equation}
where $a_j(K)$ denotes the $j$th approximation number of $K$. We note that for $A=0$ this bound reduces to a  classical eigenvalue estimate for compact operators due to K\"onig \cite{MR0482266}.  The authors of \cite{MR3296588} proved the above inequality with the help of a certain kind of perturbation determinant, which allowed them to reduce the problem to a study of the zeros of a holomorphic function.

With the results of the present paper, we will be able to broadly generalize the above estimate. Namely, we will show that a corresponding bound is true whenever $K$ is element of an $l_p$-ideal $\bi$, i.e.
\begin{equation}
  \label{eq:40}
  n_{A+K}(s) \leq   C_{p,\bi} \frac{s}{(s-\|A\|)^{p+1}} \|K\|_{\bi}^p, \qquad s> \|A\|,
\end{equation}
where $\|.\|_\bi$ denotes the quasi-norm of $\bi$ (see Corollary \ref{cor:1}). In this way we are able to extend a variety of classical eigenvalue estimates for compact operators (see, e.g., \cite{MR889455} or \cite{MR917067}) to a perturbative setting. In order to indicate the wide applicability of these new results (also to unbounded operators), in the last section of this paper we will use  (\ref{eq:40}) to obtain upper bounds on the number of discrete eigenvalues  of (perturbed) generators of $C_0$-semigroups.

The present work is meant as a self-contained introduction to regularized perturbation determinants and their application to the study of eigenvalues of perturbed operators.  As we have mentioned before, our results strongly rely on the previous works by Pietsch \cite{MR917067} and  Gohberg et al. \cite{MR1744872}. In particular, our way of constructing regularized determinants for $l_p$-ideals is an adaption of the methods of \cite{MR1744872} and might, to an expert reader, not offer much new. Indeed, we don't claim that our results on the \emph{existence} of regularized determinants are new, though it might be difficult to find them in the existing literature. On the other hand, we think that our construction of \emph{perturbation} determinants will be new even to the experts, as will be our results on the eigenvalues of perturbed operators.

Let us conclude this introduction with a short description of the contents of this paper: In the next section we will introduce some notation and discuss some preliminary results. The third section contains a discussion of regularized determinants of finite rank operators. In the first part of Section $4$ we will extend the definition of regularized determinants from finite rank operators to $l_p$-ideals and in the second part of this section we introduce and study the corresponding perturbation determinants. Section $5$ contains a discussion of various examples of $l_p$-ideals. In Section $6$, we will use our results to obtain eigenvalue estimates for perturbed operators, and in the final Section $7$ we will consider an application to generators of $C_0$-semigroups.


\section{Preliminaries}
a) Throughout this paper $X$ and $Y$ will denote complex Banach spaces and $\bdd(X,Y)$ denotes the algebra of all bounded linear operators from $X$ to $Y$.  The operator norm is denoted by $\|.\|$ (it will be clear from the context which spaces are considered). Furthermore, the finite rank operators from $X$ to $Y$ are denoted by $\ff(X,Y)$. As usual, we set $\bdd(X):=\bdd(X,X)$, et cetera. The identity operator on $X$ is denoted by $I$ and the range of $B \in \bdd(X)$ is denoted by $\ran(B)$. Finally, the dual space of $X$ is denoted by $X'$.

b) The \emph{spectrum} of a closed operator $Z$ in $X$ is defined as 
$$ \sigma(Z) := \{ \lambda \in \mc \:|\: \lambda - Z:=\lambda I -Z \text{ is not invertible in } \bdd(X)\}$$
and $\rho(Z):=\mc \setminus \sigma(Z)$ denotes its \emph{resolvent set}. The \emph{discrete spectrum} of $Z$, denoted by $\sigma_d(Z)$, consists of all isolated eigenvalues of $Z$, with finite algebraic multiplicity (these eigenvalues are also called \emph{discrete eigenvalues}). Here the algebraic multiplicity of an isolated eigenvalue $\lambda$ is given by the dimension of the corresponding Riesz projection
$$ P_Z(\lambda):= \frac 1 {2\pi i} \int_\gamma (\mu-Z)^{-1} d\mu,$$
where the contour $\gamma$ is a counterclockwise oriented circle centered at $\lambda$, with sufficiently small radius (excluding the rest of $\sigma(Z)$).  The \emph{essential spectrum} $\sigma_{ess}(Z)$ consists of all $\lambda \in \mc$ such that $\lambda-Z$ is not a Fredholm operator. It is always disjoint from the discrete spectrum of $Z$. Moreover, if $\Omega \subset \mc \setminus \sigma_{ess}(Z)$ is a connected component and $\Omega \cap \rho(Z) \neq \emptyset$, then $\Omega \cap \sigma(Z) \subset \sigma_d(Z)$. In particular, the discrete eigenvalues of $Z$ can accumulate at $\sigma_{ess}(Z)$ only. If $K$ is a compact operator on $X$, then the essential spectra of $Z$ and $Z+K$ coincide. This is usually referred to as \emph{Weyl's theorem}. As a reference for all of the previous results we refer to \cite{MR1130394}.

c) Let $\bn$ be a complex vector space. A map $\|.\|_\bn : \bn \to [0,\infty)$ is called a \emph{quasi-norm} if the following holds: (i) $\|x\|_\bn=0$ if and only if $x=0$, (ii) there exists $q \geq 1$ (the \emph{quasi-triangle constant}) such that $\|x+y\|_\bn \leq q (\|x\|_\bn+\|y\|_\bn)$
for all $x,y \in \bn$, and (iii)  $\|\lambda x\|_\bn =|\lambda| \|x\|_\bn$ for all $x \in \bn, \lambda \in \mc$. 
Clearly, if $q=1$ we speak of a norm instead of a quasi-norm. In particular, the quasi-norm induces a metrizable topology on $\bn$ such that the algebraic operations are continuous (a fundamental system of neighborhoods of zero is given by the sets $\{ x \in \bn : \|x\|_\bn \leq 1/n\},$ $n \in \mn$). In this paper, whenever we consider a quasi-normed space $(\bn,\|.\|_\bn)$ we will equip it with this metrizable topology. We say that the quasi-normed space $(\bn,\|.\|_\bn)$ is a quasi-Banach space if it is complete with respect to the corresponding metric. 
For more information on quasi-normed spaces, we refer to \cite{MR0248498} and \cite{MR917067}.

\section{Determinants for finite rank operators}

Let $F \in \ff(X)$ and $R:=\ran(F)$, so $R$ is finite-dimensional and invariant under $F$. Denoting the restriction of $F$ to $R$ by $F_R$, we can thus define
$$ \det(I-F):=\det(I_R - F_R),$$
where $I_R$ denotes the identity operator on $R$.  Similarly, we define the trace of $F$ as
$$ \tr(F):=\tr(F_R).$$

Note that if $U \subset X$ is any finite-dimensional subspace with $R \subset U$, then it is not difficult to show that  $\det(I_U-F_U)=\det(I_R-F_R)$ and $\tr(F_U)=\tr(F_R)$. Using this fact, the following properties of the determinant and trace of a finite rank operator immediately follow from the respective properties of the determinant and trace of an operator on a finite-dimensional space. For more details, see \cite{kato}, Section III.4.3.
\begin{prop}\label{prop:0}
(a)  Let $F,G \in \ff(X), B \in \bdd(X)$. Then the following holds:
  \begin{enumerate}
      \item $\det((I-F)(I-G))= \det(I-F)\det(I-G)$.\footnote{Note that $(I-F)(I-G)=I-(F+G-FG)$ and $F+G-FG \in \ff(X)$.}
      \item $\det(I-F)\neq 0$ if and only if $I-F$ is invertible in $\bdd(X)$.
      \item $\tr: \ff(X) \to \mc$ is linear.
      \item $\det(I-FB)=\det(I-BF)$ and $\tr(FB)=\tr(BF)$.
      \item $\det(I-F)=\prod_{j}(1-\lambda_j(F))$ and $\tr(F)=\sum_j \lambda_j(F)$, where $(\lambda_j(F))$ denote the discrete eigenvalues of $F$, counted according to their algebraic multiplicity.
  \end{enumerate}
(b) Let $U \subset X$ be a finite-dimensional subspace and assume that $(F_n)_{n \in \mn} \subset \ff(X)$ and $F \in \ff(X)$ satisfy $\ran(F) \subset U$ and $\ran(F_n) \subset U$ for all $n \in \mn$. If $\|F_n-F\| \to 0$ for $n \to \infty$, then 
$$ \det(I-F_n) \overset{n \to \infty}{\to} \det(I-F) \quad \text{and} \quad \tr(F_n) \overset{n \to \infty}{\to} \tr(F).$$
\end{prop}
The following result is due to Howland, see \cite{MR0417827}.
\begin{prop}\label{howland}
Let $\: \Omega \subset \mc$ be open and let $F: \Omega \to \bdd(X)$ be analytic. Suppose that $F(\lambda) \in \ff(X)$ for all $\lambda \in \Omega$. Then the mappings
$$ \Omega \ni \lambda \mapsto \det(I-F(\lambda)) \quad \text{and} \quad \Omega \ni \lambda \mapsto \tr(F(\lambda))$$
are analytic as well.
\end{prop}

Now let $p>0$. We define the \emph{$p$-regularized determinant} of $I-F, F \in \ff(X),$ as
\begin{equation}
  \label{eq:16}
  {\det}_p(I-F):= \det(I-F) \exp\left( \sum_{k=1}^{\pp - 1} \frac 1 k \tr(F^k) \right),
\end{equation}
where $\pp := \min\{ n \in \mn : n \geq p \}$ and $\sum_{k=1}^0 (\ldots):=0$. Note that  from the spectral mapping theorem and Proposition \ref{prop:0}  we obtain the alternative  representation
\begin{equation}
  \label{eq:17}
  {\det}_p(I-F) = \prod_j \left( (1-\lambda_j(F)) \exp\left( \sum_{k=1}^{\pp - 1} \frac {\lambda_j^k(F)}{k} \right) \right),
\end{equation}
where we consider the discrete eigenvalues of $F$, counted according to their algebraic multiplicity.

\begin{prop} \label{prop:1} Let $p>0$. 
  \begin{enumerate}
      \item[(a)] For $F,G \in \ff(X), B \in \bdd(X)$ the following holds:
  \begin{enumerate}
    \item ${\det}_p(I-F)\neq 0$ if and only if $I-F$ is invertible in $\bdd(X)$.
   \item ${\det}_p(I-FB)={\det}_p(I-BF)$. 
    \item If $H:=F+G-FG$, then
  \begin{equation}
    \label{eq:19}
  {\det}_p((I-F)(I-G)) = \det(I-F) {\det}_p(I-G) \exp \left( \sum_{k=1}^{\pp -1} \frac{\tr(H^k-G^k)}{k} \right).  
  \end{equation}

  \end{enumerate}
\item[(b)] Let $U \subset X$ be a finite-dimensional subspace and assume that $(F_n)_{n \in \mn} \subset \ff(X)$ and $F \in \ff(X)$ satisfy $\ran(F) \subset U$ and $\ran(F_n) \subset U$ for all $n \in \mn$. If $\|F_n-F\| \to 0$ for $n \to \infty$, then $ {\det}_p(I-F_n) \to {\det}_p(I-F)$ for $n \to \infty$.
\item[(c)] Let $\: \Omega \subset \mc$ be open and let $F: \Omega \to \bdd(X)$ be analytic. Suppose that $F(\lambda) \in \ff(X)$ for all $\lambda \in \Omega$. Then the mapping
$ \Omega \ni \lambda \mapsto {\det}_p(I-F(\lambda))$ is analytic as well.
  \end{enumerate}
\end{prop}
\begin{proof}
(a) Part (i) is a direct consequence of statement (a.ii) of Proposition  \ref{prop:0}, and part (ii) follows from statement (a.iv) of the same proposition, noting that the cyclicity of the trace implies that $\tr((FB)^k)=\tr((BF)^k)$ for all $k \in \mn$. Concerning (iii), we note that by Definition (\ref{eq:16}) and Proposition \ref{prop:0} (statements (a.i) and (a.iii)), we have
\begin{eqnarray*}
&&  {\det}_p((I-F)(I-G)) = {\det}_p(I-H) =
\det(I-H) \exp\left( \sum_{k=1}^{\pp - 1} \frac 1 k \tr(H^k) \right) \\
&=& \det(I-F) \det(I-G) \exp\left( \sum_{k=1}^{\pp - 1} \frac 1 k \tr(H^k) \right) \\
&=& \det(I-F) {\det}_p(I-G) \exp\left( \sum_{k=1}^{\pp - 1} \frac 1 k \tr(H^k-G^k) \right).
\end{eqnarray*}
Part (b) and (c) are immediate consequences of Proposition  \ref{prop:0}.b and  Proposition \ref{howland}, respectively. 
\end{proof}

We conclude this section with the following estimate. For a proof, see \cite[p.1107]{MR1009163}.
\begin{prop}
 Let $p>0$. There exists a constant $\Gamma_p > 0$ such that  for every $F \in \ff(X)$
  \begin{equation}
    \label{eq:1}
    |{\det}_{p}(I-F)| \leq \exp(\Gamma_p \sum_{k} |\lambda_k(F)|^p ),
  \end{equation}
where in the sum each discrete eigenvalue of $F$ is counted according to its algebraic multiplicity.
\end{prop}
  We emphasize that on the right-hand side of (\ref{eq:1}) we have the exponent $p$ and not the exponent $\pp$. This estimate is the main reason that we did \emph{not} restrict ourselves to integer $p$ in the definition of the regularized determinants.

A short calculation shows that $\Gamma_p \leq 1/p$ if $p \leq 1$. Moreover, for integer-valued $p \geq 2$ one has $\Gamma_p \leq (p-1)/p$ if $p \neq 3$ and $\Gamma_3 \leq 1$, see \cite{MR2391269, MR3035142}.

\section{Determinants for Quasi-normed $l_p$-ideals}

In this section we extend the definition of the regularized determinant ${\det}_p(I-L)$ from the case where the operator $L$ is of finite rank to the case where it is an element of a wider class of compact operators. 

 \begin{definition}\label{def1}
Let $p > 0$ and let $({\bi},\|.\|_{\bi})$  denote a quasi-normed subspace of $\bdd(X)$, such that the following holds:
\begin{enumerate}
    \item[(A1)] The finite rank operators $\ff(X)$ are dense in $({\bi}, \|.\|_{\bi})$.
\item[(A2)] $\|L\| \leq \|L\|_{\bi}$ for all $L \in {\bi}$.
\item[(A3)] If $L \in {\bi}$ and $A,B \in \bdd(X)$, then $ALB \in {\bi}$ and 
\begin{equation}
  \label{eq:2}
  \|ALB\|_{\bi} \leq \|A\| \|L\|_{\bi} \|B\|.
\end{equation}

    \item[(A4)] There exists a constant $\gamma_p$ such that for every $L \in {\bi}$ 
  \begin{equation}
    \label{eq:15}
\|(\lambda_j(L))\|_{l_p}:= \left( \sum_{j} |\lambda_j(L)|^p \right)^{1/p} \leq \gamma_{p} \|L\|_{\bi}.    
  \end{equation}
Here $(\lambda_j(L))$ denotes the sequence of discrete eigenvalues of $L$, counted according to their algebraic multiplicity. 
\end{enumerate}
Then $({\bi},\|.\|_{\bi})$ is called an \emph{$l_p$-ideal} in $\bdd(X)$ with \emph{eigenvalue constant} $\gamma_p$.
\end{definition}
By (A1) and (A2) every operator $L \in {\bi}$ is the $\|.\|$-limit of finite rank operators and hence is indeed compact. In particular, the non-zero spectrum of $L$ consists of discrete eigenvalues which can accumulate at zero only. For a discussion of various examples of $l_p$-ideals we refer to Section 5.
\begin{rem}\label{rem42}
The term '$l_p$-ideal' has been introduced for the purposes of this paper only and is certainly not standard. A more standard terminology for such an ideal would be an \emph{approximative embedded ideal $( \text{in } \bdd(X))$ of eigenvalue type $l_p$}. However, in this article we will stick to the shorter version. 
\end{rem}
The next theorem provides the desired extension of the regularized determinants.
 
\begin{thm}\label{thm:1} 
Let $p>0$ and let $({\bi},\|.\|_{\bi})$ be an $l_p$-ideal in $\bdd(X)$. Then there exists a unique continuous function 
${\det}_{p,{\bi}}(I-.): ({\bi}, \|.\|_{\bi}) \to \mc$ such that 
\begin{equation}
  \label{eq:9}
{\det}_{p,{\bi}}(I-F)={\det}_{p}(I-F), \quad F \in \ff(X).  
\end{equation}
Moreover, for all $K,L \in {\bi}$ 
\begin{equation}
  \label{eq:7} 
|{\det}_{p,{\bi}}(I-L)| \leq \exp(\gamma_p^p \Gamma_p \|L\|_{\bi}^p )
\end{equation}
and
\begin{eqnarray}
&& |{\det}_{p,{\bi}}(I-K)- {\det}_{p,{\bi}}(I-L)| \nonumber \\
&\leq& \|K-L\|_{\bi} \exp\left(q_{\bi}^{2p} \gamma_p^p \Gamma_p ( \|K\|_{\bi}+\|L\|_{\bi}+1)^p\right).    \label{eq:8}
\end{eqnarray}
Here $\Gamma_p$ is as given in (\ref{eq:1}),  and $q_{\bi}$ and $\gamma_p$ denote the quasi-triangle and eigenvalue constant of $({\bi},\|.\|_{\bi})$, respectively.
\end{thm}
\begin{definition}
The function ${\det}_{p,{\bi}}(I-L)$ will be called the \emph{$p$-regularized determinant of $I-L$} (with respect to ${\bi}$).  
\end{definition}
\begin{rem}
  We emphasize that it is not at all clear whether in general we have that
  \begin{equation}
    \label{eq:20}
    {\det}_{p,{\bi}}(I-L) = \prod_j \left( (1-\lambda_j(L)) \exp\left( \sum_{k=1}^{\pp - 1} \frac {\lambda_j^k(L)}{k} \right) \right), \qquad L \in \bi,
  \end{equation}
(where we use the convention that $\lambda_{N+1}(L)=\lambda_{N+2}(L)=\ldots=0$ in case there are only $N$ discrete eigenvalues). Here the problem is that we don't know whether the right-hand side, considered as a function from $(\bi, \|.\|_\bi)$ to $\mc$, is continuous. Of course, if it is continuous, then the identity holds true in view of the previous theorem and (\ref{eq:17}). We note that the right-hand side is indeed continuous if the following implication holds (see \cite{MR792906}): Whenever $(L_n)_{n \in \mn} \subset (\bi,\|.\|_\bi)$ is convergent, the sequence of partial sums $\left( \sum_{j=1}^m |\lambda_j(L_n)|^p\right)_{m \in \mn}$ is uniformly convergent in $n$. For instance, this sufficient criterion  is satisfied for operators from the Weyl-number  ideal $\mathcal{S}_p^{(x)}(X)$ (see Example \ref{ex:snumber}). For further examples, see \cite{MR792906} and \cite{MR917067}.
\end{rem}
For the proof of Theorem \ref{thm:1} we will need two lemmas. The first one is a standard result from real analysis. We leave its proof to the reader.
\begin{lemma}\label{lem:2}
  Let $(\bn,\|.\|_\bn)$ be a quasi-normed space and let $M \subset \bn$ be dense. Suppose that 
$ f : M \to \mc$ is locally uniformly continuous (with respect to the induced topology), i.e. its restriction to any ball $\{ x \in M : \|x-x_0\|_\bn < r\}, x_0 \in \bn$, $r>0,$ is uniformly continuous. Then there exists a unique continuous function $ g : (\bn,\|.\|_\bn) \to \mc$ such that $g|_M = f$.
\end{lemma}
The next lemma  extends \cite[Theorem II.4.1]{MR1744872} from normed to quasi-normed spaces. Its proof is almost literally the same as the proof of the original result. We include a sketch of the proof for completeness.
\begin{lemma}\label{lem:1}
 Let $(\bn,\|.\|_\bn)$ be a quasi-normed space with quasi-triangle constant $q_\bn$. Let $f: \bn \to \mc$ be such that
 \begin{enumerate}
     \item the function $\mc \ni \lambda \mapsto f(x+\lambda y)$ is entire for all $x,y \in \bn$, and
     \item there exists a monotonically non-decreasing function $\Theta$ on $[0,\infty)$ such that for all $x \in \bn$ we have 
$$ |f(x)| \leq \Theta(\|x\|_\bn).$$
 \end{enumerate}
Then for all $x,y \in \bn$ the following holds:
\begin{equation}
  \label{eq:10}
  |f(x)-f(y)| \leq \|x-y\|_\bn \Theta \left(q_\bn^2(\|x\|_\bn + \|y\|_\bn +1)\right).
\end{equation}
\end{lemma}
\begin{proof}
We can assume that $x \neq y$. Let $g(\lambda)= f(\frac 1 2(x+y)+\lambda(x-y))$. By assumption (i) the function $g$ is entire and so as in the proof of \cite[Theorem II.4.1]{MR1744872} we can use the mean value theorem and Cauchy's integral formula to prove that for every $\rho > 0$
\begin{equation}
  \label{eq:41}
|f(x)-f(y)| \leq \sup_{-1/2 \leq t \leq 1/2} |g'(t)| \leq \frac 1 \rho \sup_{|\lambda| \leq \rho + \frac  1 2} |g(\lambda)|.  
\end{equation}
Now for $\rho = \|x-y\|_\bn^{-1}$ and $|\lambda| \leq  \rho + \frac 1 2$ we can estimate, using that $q_\bn \geq 1$,
\begin{eqnarray*}
  \left\|\frac 1 2(x+y)+\lambda(x-y)\right\|_\bn &\leq& q_\bn \left( \frac 1 2 \|x+y\|_\bn + |\lambda|  \|x-y\|_\bn \right) \\
&\leq& q_\bn \left( \frac 1 2 (\|x+y\|_\bn + \|x-y\|_\bn)  + \rho   \|x-y\|_\bn \right) \\
&\leq& q_\bn \left( q_\bn (\|x\|_\bn + \|y\|_\bn)  +  1\right) \\
&\leq& q_\bn^2 \left(\|x\|_\bn + \|y\|_\bn  + 1 \right).
\end{eqnarray*}
By assumption (ii) this implies that for $\rho = \|x-y\|_\bn^{-1}$ and $|\lambda| \leq \rho + \frac 1 2$
$$ |g(\lambda)| \leq \Theta\left(\left\|\frac 1 2 (x+y) + \lambda (x-y) \right\|_\bn \right) \leq \Theta( q_\bn^2 \left(\|x\|_\bn + \|y\|_\bn  + 1 \right))$$
and so from (\ref{eq:41}) we obtain that
$$ |f(x)-f(y)| \leq \|x-y\|_\bn \Theta \left( q_\bn^2 \left(\|x\|_\bn + \|y\|_\bn  + 1 \right)\right).$$
\end{proof}
We are now prepared for the  proof of Theorem \ref{thm:1}.
\begin{proof}[Proof of Theorem \ref{thm:1}]
Let us define $f: (\ff(X), \|.\|_{\bi}) \to \mc$ by $ f(F)= {\det}_p (I-F)$.  
Then from estimate (\ref{eq:1}) and assumption (A4) we obtain that
\begin{equation}
  \label{eq:11}
 |f(F)| \leq \exp( \Gamma_p \sum_k |\lambda_k(F)|^p) \leq \exp( \Gamma_p \gamma_p^p \|F\|_{\bi}^p), \quad F \in \ff(X).  
\end{equation}
Furthermore, Proposition \ref{prop:1}.c implies that $\mc \ni \lambda \mapsto f(F+\lambda G)$ is an entire function for all $F,G \in \ff(X)$. We can hence apply Lemma \ref{lem:1} (with $\Theta(r)=\exp(\gamma_p^p \Gamma_p r^p))$ to obtain that for all $F,G \in \ff(X)$
\begin{equation}
  \label{eq:12}
  |f(F)-f(G)| \leq \|F-G\|_{\bi} \exp\left(\gamma_p^p \Gamma_p  q_{\bi}^{2p} (\|F\|_{\bi} + \|G\|_{\bi} +1)^p \right).
\end{equation}
In particular, this estimate shows that $f : (\ff(X), \|.\|_{\bi}) \to \mc$ is locally uniformly continuous. Since $\ff(X)$ is dense in $({\bi},\|.\|_{\bi})$ by assumption (A1),  Lemma \ref{lem:2} implies that there exists a unique continuous function ${\det}_{p,{\bi}}(I-.) : ({\bi}, \|.\|_{\bi}) \to \mc$ such that ${\det}_{p,{\bi}}(I-F) = f(F)={\det}_p(I-F)$ for $F \in \ff(X)$. The validity of (\ref{eq:7}) and (\ref{eq:8}) is an immediate consequence of (\ref{eq:11}),(\ref{eq:12}),(A1),(A2) and the continuity of ${\det}_{p,{\bi}}(I-.)$.
\end{proof}
In the following proposition we gather some important properties of the $p$-regularized determinant.
\begin{prop}\label{prop:4}
Let $p>0$ and let $({\bi},\|.\|_{\bi})$ be an $l_p$-ideal in $\bdd(X)$. Let $L \in {\bi}, F \in \ff(X)$  and $B \in \bdd(X)$. Then the following holds:
\begin{enumerate}
    \item ${\det}_{p,{\bi}}(I-LB)={\det}_{p,{\bi}}(I-BL)$.
    \item If $H:=F+L-FL$, then 
  \begin{small}
\begin{eqnarray}
      \label{eq:18}
&&  {\det}_{p,{\bi}}((I-F)(I-L)) \nonumber \\
&=& \det(I-F) {\det}_{p,{\bi}}(I-L) 
\exp \left( \sum_{k=1}^{\pp-1} \sum_{m=0}^{k-1} \frac{\tr(F(I-L)L^mH^{k-1-m})}{k} \right).
\end{eqnarray}      
  \end{small}
\item ${\det}_{p,{\bi}}(I-L) \neq 0$ if and only if $I-L$ is invertible in $\bdd(X)$.
\end{enumerate}
\end{prop}
Concerning (ii) we note that $(I-F)(I-L)=I-H$ and that $H \in {\bi}$.
\begin{proof}
(i) By (A3) we have $LB,BL \in {\bi}$. Moreover, by (A1) there exists $(L_n)_{n \in \mn} \subset \ff(X)$ with $\|L-L_n\|_{\bi} \to 0$ for $n \to \infty$, so by (A3) $ \|L_nB - LB \|_{\bi} \to 0$ and $\|BL_n - BL\|_{\bi} \to 0$ for $n \to \infty$ as well. Using statement (a.ii) of Proposition \ref{prop:1} we thus obtain
$$ {\det}_{p,{\bi}}(I-LB)= \lim_n {\det}_{p}(I-L_nB) = \lim_n {\det}_{p}(I-BL_n)={\det}_{p,{\bi}}(I-BL).$$
(ii) Again, let $(L_n)_{n \in \mn} \subset \ff(X)$ with $\|L-L_n\|_{\bi} \to 0$ for $n \to \infty$. Then by (A3) also $\|H-H_n\|_{\bi} \to 0$ for $n \to \infty$, where $H_n:= F+L_n-FL_n \in \ff(X)$. Hence, using (\ref{eq:19}) we obtain
\begin{eqnarray*}
&&  {\det}_{p,{\bi}}((I-F)(I-L)) = {\det}_{p,{\bi}}(I-H) \\
&=& \lim_n {\det}_{p}(I-H_n) = \lim_n {\det}_{p}((I-F)(I-L_n)) \\
&=& \lim_n \det(I-F) {\det}_p(I-L_n) \exp \left( \sum_{k=1}^{\pp-1} \frac{\tr(H_n^k-L_n^k)}{k} \right).
\end{eqnarray*}
We already know that $\lim_n {\det}_p(I-L_n)={\det}_{p,{\bi}}(I-L)$. Let us show that
$$ \lim_n \exp \left( \sum_{k=1}^{\pp-1} \frac{\tr(H_n^k-L_n^k)}{k} \right)$$
exists as well. To this end, note that for every $1 \leq k \leq \pp -1$ we have
$$ H_n^k-L_n^k = \sum_{m=0}^{k-1} H_n^{k-1-m}(H_n-L_n)L_n^m=\sum_{m=0}^{k-1} H_n^{k-1-m}F(I-L_n)L_n^m,$$
so by the linearity and cyclicity of the trace:
$$ \tr(H_n^k-L_n^k) = \sum_{m=0}^{k-1} \tr(F(I-L_n)L_n^mH_n^{k-1-m}).$$
The ranges of the operators $F(I-L_n)L_n^mH_n^{k-1-m}$ are all contained in $\ran(F)$ and 
$$ \|F(I-L_n)L_n^mH_n^{k-1-m}-F(I-L)L^mH^{k-1-m}\| \to 0 \quad (n \to \infty)$$
by (A2) and the fact that $H_n \overset{{\bi}}{\to} H$ and $L_n \overset{{\bi}}{\to} L$. Hence, Proposition \ref{prop:0}.b implies that 
$$ \lim_n \tr(H_n^k-L_n^k) = \sum_{m=0}^{k-1} \tr(F(I-L)L^mH^{k-1-m})$$ 
and so 
$$ \lim_n \exp \left( \sum_{k=1}^{\pp-1} \frac{\tr(H_n^k-L_n^k)}{k} \right) 
= \exp \left( \sum_{k=1}^{\pp-1} \sum_{m=0}^{k-1} \frac{\tr(F(I-L)L^mH^{k-1-m})}{k} \right),$$
concluding the proof of (ii).

(iii) Since $(\bi, \|.\|_\bi) \ni K \mapsto {\det}_{p,{\bi}}(I-K) \in \mc$ is  continuous and ${\det}_{p,{\bi}}(I)={\det}_{p,{\bi}}(I-0)=1,$ there exists some $1>\delta > 0$ such that ${\det}_{p,{\bi}}(I-K) \neq 0$ if $\|K\|_{\bi}<\delta$.  

By (A1) and (A2) we can find $F \in \ff(X)$ and $G \in {\bi}$ such that $\|G\| \leq \|G\|_{\bi}< \delta$ and $L=F+G$. In particular, this implies that 
$ {\det}_{p,{\bi}}(I-G) \neq 0$ and that $I-G$ is invertible in $\bdd(X)$. Writing 
$$ I-L=(I-F(I-G)^{-1})(I-G)$$
and using part (ii) of the present proposition with $F'=F(I-G)^{-1}, L'=G$ and $H'=F'+L'-F'L'=L$, we obtain that 
\begin{eqnarray*}
&&{\det}_{p,{\bi}}(I-L)  \\
&=& {\det}(I-F(I-G)^{-1}) {\det}_{p,{\bi}}(I-G)   \exp 
\left( \sum_{k=1}^{\pp-1} \sum_{m=0}^{k-1} \frac{\tr(F G^m L^{k-1-m})}{k} \right).
\end{eqnarray*}
With our above preparations we can thus conclude that ${\det}_{p,{\bi}}(I-L) \neq 0$ if and only if $\det(I-F(I-G)^{-1})\neq 0$, which is the case if and only if $I-F(I-G)^{-1}=(I-L)(I-G)^{-1}$ is invertible in $\bdd(X)$. So ${\det}_{p,{\bi}}(I-L) \neq 0$ if and only if $I-L$ is  invertible in $\bdd(X)$.
\end{proof}

We are finally prepared for the introduction of regularized perturbation determinants.
\begin{definition}
 Let $p>0$ and let $({\bi},\|.\|_{\bi})$ be an $l_p$-ideal in $\bdd(X)$. For $A \in \bdd(X)$ and $K \in {\bi}$ the product $K(\lambda-A)^{-1}, \lambda \in \rho(A),$ is an element of ${\bi}$ by (A3). So the function
\begin{equation}
  \label{eq:4}
 D=D^{A,K}_{p,{\bi}}: \rho(A) \ni \lambda \mapsto {\det}_{p,{\bi}}(I-K(\lambda-A)^{-1})  
\end{equation}
is well defined. It is called the \emph{p-regularized perturbation determinant} of $A$ by $K$ (with respect to ${\bi}$).  
\end{definition}
The following theorem is the main result of this section.

\begin{thm}\label{thm:2}
Let $p>0$ and let $({\bi},\|.\|_{\bi})$ be an $l_p$-ideal in $\bdd(X)$ with eigenvalue constant $\gamma_p$. Then the following holds:
\begin{enumerate}
    \item $\lim_{|\lambda| \to \infty} D(\lambda)=1$.
    \item $D$ is analytic on $\rho(A)$.
    \item For $\lambda \in \rho(A)$ we have
  \begin{equation}
    \label{eq:3}
    |D(\lambda)| \leq \exp \left( \gamma_p^p \Gamma_p \|K(\lambda-A)^{-1}\|_{\bi}^p \right),
  \end{equation}
where $\Gamma_p$ is as given in (\ref{eq:1}).
    \item $D(\lambda)= 0$ if and only if $\lambda \in \sigma(A+K)$.
    \item Let $\lambda_0 \in \rho(A) \cap \sigma_d(A+K)$. Then its algebraic multiplicity as an eigenvalue of $A+K$ coincides with its order as a zero of $D$.
\end{enumerate}
\end{thm}
\begin{proof}
(i) This follows from the fact that ${\det}_{p,\bi}(I)=1$, the continuity of the $p$-regularized determinant and the fact that, by (A3), $K(\lambda-A)^{-1} \overset{{\bi}}{\to} 0$ for $|\lambda| \to \infty$.

(ii) By (A1) there exists a sequence $(K_n) \subset \ff(X)$ with $\|K_n-K\|_{\bi} \to 0$ for $n \to \infty$. Let $\Omega \subset \rho(A)$ be any compact set. Then $\dist(\Omega, \sigma(A)) > 0$ and 
$$ s_A(\Omega) := \sup_{\lambda \in \Omega} \|(\lambda-A)^{-1}\| < \infty.$$
In particular, using (A3) we obtain that
$$ \|K(\lambda-A)^{-1}-K_n(\lambda-A)^{-1}\|_{\bi}\leq \|K-K_n\|_{\bi} \|(\lambda-A)^{-1}\| \leq s_A(\Omega) \|K-K_n\|_{\bi}$$
for $\lambda \in \Omega$, which shows that 
$$ \sup_{\lambda \in \Omega} \|K_n(\lambda-A)^{-1} -K(\lambda-A)^{-1}\|_{\bi} \to 0.$$
Since this is true for any compact $\Omega \subset \rho(A)$, estimate (\ref{eq:8}) implies that 
$${\det}_{p,{\bi}}(I-K_n(\lambda-A)^{-1}) \overset{n \to \infty}{\to} {\det}_{p,{\bi}}(I-K(\lambda-A)^{-1})=D(\lambda)$$
compactly on $\rho(A)$. But the function $${\det}_{p,{\bi}}(I-K_n(\lambda-A)^{-1}) = {\det}_{p}(I-K_n(\lambda-A)^{-1})$$ is analytic on $\rho(A)$ by Proposition \ref{prop:1}.c.

 (iii) This follows immediately from estimate (\ref{eq:7}).

(iv) From statement (iii) of Proposition \ref{prop:4} we know that $D(\lambda)=0$ if and only if $I-K(\lambda-A)^{-1}=(\lambda-(A+K))(\lambda-A)^{-1}$ is not invertible in $\bdd(X)$. This is the case if and only if $\lambda \in \sigma(A+K)$.

(v) Let $\lambda_0 \in \rho(A) \cap \sigma_d(A+K)$. It is no restriction to assume that $\lambda_0 \neq 0$. By part (iv) we know that $D(\lambda_0)=0$. Let us denote the Riesz projection of $B:=A+K$ with respect to $\lambda_0$ by $P$, and set $T:=BP$ and $T^\perp:=B(I-P)$. Note that $T$ is of finite rank, with $\sigma(T)=\{\lambda_0\}$, and that $\lambda_0 \notin \sigma(T^\perp)$. In particular, there exists a ball $U_r(\lambda_0)$ around $\lambda_0$ such that $0 \notin U_r(\lambda_0)$ and such that $\lambda-A$ and $\lambda-T^\perp$ are invertible for all $\lambda \in U_r(\lambda_0)$. Now a short computation, using $TT^\perp=T^\perp T=0$ and $B=T+T^\perp$, shows that for $\lambda \in U_r(\lambda_0)$
$$ I-K(\lambda-A)^{-1} = (I-\lambda^{-1}T)(I-(T^\perp-A)(\lambda-A)^{-1}).$$ 
Hence, by Proposition \ref{howland} and statement (ii) of Proposition \ref{prop:4} there exists a holomorphic function $F_{B,A,p} : U_r(\lambda_0) \to \mc$ such that for $\lambda \in U_r(\lambda_0)$
$$ D(\lambda) = {\det}(I-\lambda^{-1}T){\det}_{p, {\bi}}(I-(T^\perp-A)(\lambda-A)^{-1})\exp(F_{B,A,p}(\lambda)).$$
The operator $I-(T^\perp-A)(\lambda-A)^{-1}=(\lambda-T^\perp)(\lambda-A)^{-1}$ is invertible for $\lambda \in U_r(\lambda_0)$, so from statement (iii) of Proposition \ref{prop:4} it follows that the multiplicity of $\lambda_0$ as a zero of $D$ coincides with its order as a zero of $\lambda \mapsto {\det}(I-\lambda^{-1}T) = (1-\lambda^{-1}\lambda_0)^{\rank(P)}$. 
But the rank of $P$ coincides with the algebraic multiplicity of $\lambda_0$ as an eigenvalue of $B=A+K$.
\end{proof}

\section{Examples of $l_p$-ideals}\label{sec:examples}

In this section we will discuss several examples of $l_p$-ideals. We start with the most important one.
 \begin{ex} (Quasi-Banach operator ideals of eigenvalue type $l_p$ (in the sense of Pietsch))
Let $\bdd$ denote the class of all bounded linear operators between Banach spaces, i.e. $ \bdd := \cup_{X,Y} \bdd(X,Y).$ A subclass $\ac \subset \bdd$ together with a mapping $\alpha : \ac \to \mr_+$ is called a \emph{quasi-Banach operator ideal} (in the sense of Pietsch) if for all Banach spaces $X,Y$ the components $\ac(X,Y):= \ac \cap \bdd(X,Y)$ have the following properties:
\begin{enumerate}
      \item $\ff(X,Y) \subset \ac(X,Y)$ and $\alpha(I_\mc)=1$. Here $I_\mc : \mc \to \mc$ denotes the identity operator.
    \item $(\ac(X,Y), \alpha)$ is a quasi-Banach space with a quasi-triangle constant $q_\ac$ independent of $X$ and $Y$.
 \item If $L \in \ac(X,Y), A \in \bdd(Y,Y_0), B \in \bdd(X_0,X)$ for Banach spaces $X_0,Y_0$, then $ALB \in \ac(X_0,Y_0)$ and $\alpha(ALB) \leq \|A\| \alpha(L) \|B\|$.
\end{enumerate}
We note that (i) - (iii) also imply that 
\begin{enumerate}
    \item[(iv)] $\|L\| \leq \alpha(L)$ for all $L \in \ac(X,Y)$ and all Banach spaces $X,Y$. 
\end{enumerate}
If all components $(\ac(X,Y),\alpha)$ are Banach spaces, then $(\ac, \alpha)$ is called a \emph{Banach operator ideal}. A quasi-Banach operator ideal $(\ac, \alpha)$ is called \emph{approximative}, if $\ff(X,Y)$ is dense in $(\ac(X,Y),\alpha)$ for all Banach spaces $X,Y$. We note that for every quasi-Banach operator ideal $\ac$ the class
$$\ac^0 := \cup_{X,Y} \ac^0(X,Y):= \cup_{X,Y} \overline{\ff(X,Y)}^\alpha,$$
where $\overline{(.)}^\alpha$ denotes the closure with respect to $\alpha$, is an approximative quasi-Banach operator ideal, called the \emph{approximative kernel} of $\ac$. Finally, the quasi-Banach operator ideal $\ac$ is said to be \emph{of eigenvalue type $l_p$}, if, for all Banach spaces $X$, all operators $L \in \ac(X):=\ac(X,X)$ have a compact power and their eigenvalue sequence $(\lambda_n(L))$ is in $l_p(\mn)$. If this is the case, there will exist a constant $c_p \geq 1$ such that for all Banach spaces $X$ and all $L \in \ac(X)$ we have
\begin{equation}
  \label{eq:6}
\|(\lambda_j(L))\|_{l_p} \leq c_{p} \alpha(L). 
\end{equation}
For all these results, and much more, we refer to Pietsch's monographs \cite{MR582655, MR917067}. 

We can now conclude that if $\ac$ is an approximative quasi-Banach operator ideal of eigenvalue type $l_p$ (satisfying (\ref{eq:6})),  then $(\ac(X),\alpha)$ is an $l_p$-ideal in $\bdd(X)$, with eigenvalue constant $c_p$, for all Banach spaces $X$.
 \end{ex}
\begin{rem}
Most of the  more specific $l_p$-ideals considered in the following examples arise as components of some approximative quasi-Banach operator ideal $\ac$ as discussed above. Hence, one might well ask why we didn't restrict our attention to such ideals from the start. Our answer to this question is twofold: first, we will discuss at least one example which does not exactly fit into this scheme (the entropy number quasi-ideal). More importantly, we preferred to make only those assumptions that were absolutely necessary for our construction of a well-behaved regularized perturbation determinant (for instance, we didn't need that an $l_p$-ideal is complete).
\end{rem}
For details on the following examples we refer to the books by Pietsch \cite{MR917067} and K\"onig \cite{MR889455}, and references therein.
\begin{ex} ($p$-summing operators) \label{ex:summing} 
Let $p \geq 1$. An operator $L \in \bdd(X,Y)$ is called \emph{$p$-summing} if there exists $c>0$ such that for an arbitrary choice of finitely many $x_1,\ldots,x_n \in X$ we have
\begin{equation}
  \label{eq:5}
  \left( \sum_{j=1}^n \|Lx_j\|_Y^p \right)^{1/p} \leq c \sup\left\{ \left( \sum_{j=1}^n |x'(x_j)|^p \right)^{1/p}: x' \in X', \|x'\|_{X'}=1\right\}. 
\end{equation}
The smallest constant $c$ with this property is called the $p$-summing norm of $L$, written as $\|L\|_{\Pi_p}$. The class $(\Pi_p, \|.\|_{\Pi_p})$ of all $p$-summing operators between Banach spaces is a Banach operator ideal of eigenvalue type $l_q$, where $q:=\max(2,p)$, and
$$ \|(\lambda_j(L))\|_{l_q} \leq \|L\|_{\Pi_p}$$
for every $L \in \Pi_p(X)$ and all Banach spaces $X$. Note that this ideal is not approximative (for instance, there exist non-compact  $p$-summing operators). We conclude that for a generic Banach space $X$ the approximative kernel $(\Pi^0_p(X),\|.\|_{\Pi_p})$ is  an $l_q$-ideal with eigenvalue constant $1$. 

\begin{rem}\label{rem:persson}
For more specific Banach spaces $X$, one can show that the finite rank operators are dense in $(\Pi_p(X),\|.\|_{\Pi_p})$. For instance, if the dual space $X'$ has the approximation property and, in addition, is either reflexive or separable, then $(\Pi^0_p(X), \|.\|_{\Pi_p})=(\Pi_p(X), \|.\|_{\Pi_p})$, as follows from the results in \cite{MR0247504} and \cite{MR0243323}.    
\end{rem}
\end{ex}
\begin{ex} (nuclear operators)
An operator $L \in \bdd(X,Y)$ is called \emph{nuclear} if there exists a representation
\begin{equation}
  \label{eq:13}
  Lx = \sum_j x_j'(x)y_j, \quad x \in X,
\end{equation}
where $x_j' \in X'$ and $y_j \in Y$ satisfy $\sum_j \|x_j'\|_{X'} \|y_j\|_Y < \infty$. If this is the case we write $L \in \operatorname{Nuc}(X,Y)$ and define 
$$\|L\|_{\operatorname{Nuc}}:= \inf \sum_j \|x_j'\|_{X'} \|y_j\|_Y,$$
the infimum being taken over all representations of the form (\ref{eq:13}). The class $(\operatorname{Nuc},\|.\|_{\operatorname{Nuc}})$ of all nuclear operators between Banach spaces is an approximative Banach operator ideal of eigenvalue type $l_2$ and
$$ \|(\lambda_j(L))\|_{l_2} \leq \|L\|_{\operatorname{Nuc}}$$
for every $L \in \operatorname{Nuc}(X)$ and every Banach space $X$. In particular, $(\operatorname{Nuc}(X),\|.\|_{\operatorname{Nuc}})$ is an $l_2$-ideal with eigenvalue constant $1$.
\end{ex}
\begin{rem}
Let $(\Omega,\mu)$ be a $\sigma$-finite measure space and set $X=L_q(\Omega,\mu)$ for some $q \geq 1$. Then $(\operatorname{Nuc}(L_q(\Omega,\mu)),\|.\|_{\operatorname{Nuc}})$ is an $l_p$-ideal with eigenvalue constant $1$, where $1/p= 1 -|1/q-1/2|$, see \cite[Theorem 22]{MR1863710}.
\end{rem}
\begin{ex} (s-number ideals) \label{ex:snumber}
A map $s: \bdd \to l_\infty(\mn)$ assigning to every bounded operator $A \in \bdd$ a sequence $(s_n(A))_{n \in \mn} \subset \mr_+$ is called an s-number function if the following holds: 
\begin{enumerate}
    \item $\|A\| =s_1(A) \geq s_2(A) \geq \ldots \geq 0$ for all $A \in \bdd(X,Y)$.
    \item $s_{n+m-1}(A+B) \leq s_n(A) + s_m(B)$ for all $A,B \in \bdd(X,Y), n,m \in \mn$.
    \item $s_n(ABC) \leq \|A\| s_n(B) \|C\|$ for all $A \in \bdd(Y,Y_0),B \in \bdd(X,Y), C \in \bdd(X_0,X)$.
    \item $s_n(T)=0$ if $\rank(T) < n$ and $s_n(I_{l_2^n})=1$, where $l_2^n=l_2(\{1,\ldots,n\})$. 
\end{enumerate}
The sequence $(s_n(A))_{n \in \mn}$ is called the sequence of s-numbers of $A$. Now for $p > 0$ and Banach spaces $X,Y$ define
\begin{equation}
  \label{eq:14}
  \mathcal{S}_{p}^{(s)}(X,Y):= \{ A \in \bdd(X,Y) : \|A\|_{p}^{(s)}:= \|(s_n(A))\|_{l_p} < \infty\}.
\end{equation}
Then the class $(\mathcal{S}_{p}^{(s)}, \|.\|_{p}^{(s)})$ is a quasi-Banach operator ideal. 

We note that if $X=Y=\hil$ is a complex Hilbert space, every $s$-number sequence coincides with the sequence of singular numbers and hence $\mathcal{S}_{p}^{(s)}(\hil)=S_p(\hil)$, the von Neumann-Schatten class (see the introduction). However, for general Banach spaces there exists a multitude of $s$-number sequences, the most important ones being the \emph{approximation numbers}
\begin{equation}
  \label{eq:22}
  a_n(A)= \inf \{ \|A-F_n\| : F_n \in \ff(X,Y), \rank(F_n) < n \},
\end{equation}
the \emph{Gelfand numbers}
\begin{equation}
  \label{eq:23}
  c_n(A)= \inf \{ \|A_{X_n}\| : X_n \subset X \text{ has codimension } < n \text{ in } X \}
\end{equation}
and the \emph{Weyl numbers}
\begin{equation}
  \label{eq:24}
  x_n(A)= \sup\{ a_n(AB) : B \in \bdd(l_2^n, X), \|B\| \leq 1\}.
\end{equation}
We remark that $x_n(A) \leq c_n(A) \leq a_n(A), n \in \mn,$ and so
$$ S_p^{(a)}(X,Y) \subset S_p^{(c)}(X,Y) \subset S_p^{(x)}(X,Y).$$
Moreover, the following holds: Let $s \in \{a,c,x\}$. Then $(\mathcal{S}_p^{(s)},\|.\|_p^{(s)})$ is a quasi-Banach operator ideal of eigenvalue type $l_p$ and
$$ \|(\lambda_j(L))\|_{l_p} \leq 2^{1/p} \sqrt{2e} \|L\|_p^{(s)}, \qquad L \in \mathcal{S}_p^{(s)}(X).$$
In addition, the approximation- and Gelfand number ideals are approximative. In particular, we obtain that $(S_p^{(a)}(X),\|.\|_p^{(a)}), (S_p^{(c)}(X),\|.\|_p^{(c)})$ and $((S_p^{(x)})^0(X),\|.\|_p^{(x)})$ are $l_p$-ideals with eigenvalue constant $2^{1/p} \sqrt{2e}$.
\end{ex}  
\begin{rem}
If the Banach space $X$ has the bounded approximation property, then the finite rank operators are dense in $(S_p^{(x)}(X),\|.\|_p^{(x)})$ and so $((S_p^{(x)})^0(X),\|.\|_p^{(x)})=(S_p^{(x)}(X),\|.\|_p^{(x)})$, see \cite{MR889455}, page 220. 
\end{rem}

\begin{ex} (entropy number ideal)
 For $L \in \bdd(X,Y)$ and $n \in \mn$ let $e_n(L)$ denote the infimum of all $\eps>0$ such that there exist $y_1, \ldots, y_q \in Y, q \leq 2^{n-1},$ with the property that 
$$ L(B_X) \subset \cup_{j=1}^q ( \{y_j\} + \eps B_Y).$$
Here $B_X,B_Y$ denote the closed unit balls in $X$ and $Y$, respectively. The sequence $(e_n(L))_{n \in \mn}$ is called the sequence of entropy numbers of $L$. We note that it satisfies properties (i)-(iii) of an $s$-number sequence, but not property (iv).

For $p>0$ define 
\begin{equation}
  \label{eq:25}
  \mathcal{S}_{p}^{(e)}(X,Y):= \{ L \in \bdd(X,Y) : \|L\|_{p}^{(e)}:= \|(e_n(L))\|_{l_p} < \infty.\}.
\end{equation}
Then for every Banach space $X$ the space $(\mathcal{S}_p^{(e)}(X), \|.\|_p^{(e)})$ is an $l_p$-ideal in $\bdd(X)$, see \cite{MR1098497}. We note that the class $(\mathcal{S}_p^{(e)}, \|.\|_p^{(e)})$ is not a quasi-Banach operator ideal in the sense of Pietsch, since it does not satisfy $\|I_\mc\|_p^{(e)}=1$.   
\end{ex}

For some results on the relations between the different $l_p$-ideals discussed in this section we refer to \cite{MR1098497}, \cite{MR889455}, \cite{MR582655}, \cite{MR917067} and \cite{MR2300779}. Moreover, we note that for specific classes of operators (like integral- or embedding operators) there exist various criteria that allow us to check whether a given operator is an element of one of the $l_p$-ideals described above. We will provide one such criterion in Example \ref{ex:hille} below. For  much more on this topic, see \cite{MR889455} and  \cite{MR917067}.
 
\section{Eigenvalues of compactly perturbed operators}

 In this section we consider a bounded operator $A \in \bdd(X)$ and a perturbation $K \in {\bi}$, where ${\bi}$ is a fixed $l_p$-ideal in $\bdd(X)$, $p>0$, with eigenvalue constant $\gamma_p$. We are interested in the discrete spectrum of the perturbed operator $A+K$.
 \begin{rem} 
Our approach in this section follows along the lines of the approach developed in \cite{MR3296588}. However, we note that in \cite{MR3296588} the authors considered the case where ${\bi}=\mathcal{S}_p^{(a)}(X)$ is the approximation number ideal only. 
 \end{rem}
Let us first note that by Weyl's theorem on the preservation of the essential spectrum under compact perturbations, we have $\sigma_{ess}(A)=\sigma_{ess}(A+K)$. To obtain some information on the discrete spectrum of $A+K$ as well, we will use the perturbation determinant
\begin{equation}
  \label{eq:26}
D=D^{A,K}_{p,{\bi}}: \rho(A) \ni \lambda \mapsto {\det}_{p,{\bi}}(I-K(\lambda-A)^{-1}),  \end{equation}
which was introduced and studied above. We start with a simple lemma.
\begin{lemma}\label{lem:simple}
 Let $\Omega \subset \rho(A)$ be open and connected with $\Omega \cap \rho(A+K) \neq \emptyset$. Then the following holds:
 \begin{enumerate}
     \item $\Omega \cap \sigma(A+K) \subset \sigma_d(A+K)$.
     \item Let $\lambda \in \Omega$. Then $\lambda \in \sigma_d(A+K)$ if and only if $ D(\lambda)= 0$, and in this case the algebraic multiplicity of $\lambda$ as a discrete eigenvalue of $A+K$ coincides with its order as a zero of $D$. 
  \end{enumerate}
\end{lemma}
\begin{proof}
  (i) Since $\sigma_{ess}(A)=\sigma_{ess}(A+K)$ and $\Omega \subset \rho(A)$, we have $\Omega \subset \mc \setminus \sigma_{ess}(A+K)$. As $\Omega \cap \rho(A+K) \neq \emptyset$ by assumption, this implies that $\Omega \cap \sigma(A+K) \subset \sigma_d(A+K)$, see, e.g., \cite[Theorem 2.1 on page 373]{MR1130394}.

(ii) This is an immediate consequence of (i) and Theorem \ref{thm:2}.
\end{proof}
Let $\eps > 0$. We recall that the \emph{$\eps$-pseudospectrum} of $A$ is defined as 
$$ \sigma_\eps(A)=\{ \lambda \in \mc : \|(\lambda-A)^{-1}\| > 1/\eps\}.$$
\begin{rem}
In the last two decades, the $\eps$-pseudospectrum has been studied extensively, both from an analytical and a numerical perspective, see, e.g., \cite{b_Davies} and \cite{MR2155029} and references therein. We note that the set
$ \{ \lambda \in \mc: \dist(\lambda,\sigma(A)) < \eps\}$ is contained in $\sigma_\eps(A)$ but, in general, the $\eps$-pseudospectrum can be much larger.
\end{rem}
In the following we are going to establish upper bounds on the number of discrete eigenvalues of $A+K$ in certain subsets of the complement of $\sigma_{\eps}(A)$. To this end, let $\hat{\mc}=\mc \cup \{\infty\}$ denote the extended complex plane. For a simply connected subset $\Omega \subset \hat \mc \setminus \sigma_\eps(A)$, with $\infty \in \Omega$, let $\phi : \Omega \to \md$ denote a conformal mapping with $\phi(\infty)=0$. Moreover, for a subset $\Omega' \subset \Omega$ define
\begin{equation}
  \label{eq:42}
r_{\Omega}(\Omega'):= \sup_{z \in \Omega'} |\phi(z)|.  
\end{equation}
Here the values of $r_{\Omega}$ do not depend on the choice of $\phi$, since all such mappings differ only by a unimodular constant. Moreover, we note that $0 \leq r_{\Omega}(\Omega') \leq 1$ and that $r_{\Omega}(\Omega')=0$ if and only if $\Omega'=\{\infty\}$.

Finally, let us denote the number of discrete eigenvalues of $A+K$ in $\Omega'$ (counting algebraic multiplicity) by $\mathcal{N}_{A+K}(\Omega')$.
\begin{thm}\label{thm:3}
Let $\Omega \subset \hat{\mc} \setminus \sigma_\eps(A)$ be simply connected with $\infty \in \Omega$. Then for $\Omega' \subset \Omega$ with $0<r_\Omega(\Omega')<1$ the following holds:
\begin{eqnarray*}
\mathcal{N}_{A+K}(\Omega') \leq \frac{\gamma_p^p \Gamma_p}{ \eps^p \log( \frac 1 {r_\Omega(\Omega')})} \|K\|_{\bi}^p.
\end{eqnarray*} 
Here $\gamma_p$ denotes the eigenvalue constant of ${\bi}$ and $\Gamma_p$ is as given in (\ref{eq:1}).
\end{thm}
The previous estimate broadly generalizes a corresponding result in \cite{MR3296588}, which considered the case where ${\bi}=\mathcal{S}_p^{(a)}(X)$.
\begin{proof}[Proof of Theorem \ref{thm:3}]
Let $\phi : \Omega \to \md$ be conformal with $\phi(\infty)=0$. Consider the function $h= D \circ \phi^{-1}$ defined on $\md$. From Theorem \ref{thm:2} we obtain that $h$ is holomorphic with $h(0)=D(\infty)=1$, and so Jensen's identity (see, e.g., \cite{MR924157}) implies that for $0<r<1$
$$ \int_0^r \frac{n(h;s)}{s} ds = \frac 1 {2\pi} \int_0^{2\pi} \log |h(re^{it})| dt,$$
where $n(h;s)$ denotes the number of zeros of $h$ in $\{ w  : |w| \leq s\}$ (counting order). We can now estimate 
$$ n(h;r) \log \frac 1 r = \int_r^1 \frac{n(h;r)} s ds \leq \int_r^1 \frac{n(h;s)} s ds \leq \int_0^1 \frac{n(h;s)} s ds \leq \log \|h\|_\infty,$$
where we used Jensen's identity to obtain the last inequality. By Lemma \ref{lem:simple}, every discrete eigenvalue of ${A+K}$ in $\Omega'$ corresponds to a zero of $D$, hence to a zero of $h$ in $\phi(\Omega')$. But $\phi(\Omega')$ is a subset of the disk $\{ w : |w| \leq r_\Omega(\Omega')\}$, so from the previous inequality we obtain that
$$ \mathcal{N}_{A+K}(\Omega') \leq n(h;r_\Omega(\Omega')) \leq \frac{\log \|h\|_\infty}{\log \frac 1 {r_\Omega(\Omega')}}.$$
It remains to observe that from statement (iii) in Theorem \ref{thm:2} we also know that
\begin{eqnarray*}
 \log \|h\|_\infty &=& \sup_{\lambda \in \Omega} \log|D(\lambda)| 
\leq \gamma_p^p \Gamma_p  \sup_{\lambda \in \Omega} \|K(\lambda-A)^{-1}\|_\bi^p \\
&\leq& \gamma_p^p \Gamma_p \|K\|_{\bi}^p \sup_{\lambda \in \Omega} \|(\lambda-A)^{-1}\|^p  \leq \eps^{-p} \gamma_p^p \Gamma_p \|K\|_{\bi}^p.
\end{eqnarray*}
Here in the next to last inequality we used (A3) and in the last inequality we used the fact that $\Omega \subset \hat{\mc} \setminus \sigma_{\eps}(A)$.
\end{proof}
The previous theorem is very general, but also quite abstract. In order to make it more explicit we have at least two choices: (i) We can consider more specific operators $A$ whose pseudospectrum $\sigma_\eps(A)$ can be computed explicitly (and then apply the theorem with correspondingly chosen sets $\Omega$ and $\Omega'$), or (ii) we can consider a general $A$ but restrict our attention to simple choices of $\Omega$ and $\Omega'$ (which will allow the computation of $r_\Omega(\Omega')$). In this paper, we will stick to the second choice. For instance, in the following theorem we will provide a bound on 
\begin{equation}
  \label{eq:28}
n_{A+K}(s):= \mathcal{N}_{A+K}(\{ \lambda: |\lambda|>s\}),
\end{equation}
the number of discrete eigenvalues of ${A+K}$ outside a ball of radius $s > r(A)$. Here 
\begin{equation}
  \label{eq:43}
r(A)= \sup\{ |\lambda| : \lambda \in \sigma(A)\}  
\end{equation}
denotes the \emph{spectral radius} of $A$. As we will see, to obtain explicit results we will need to restrict ourselves to complements of balls where we have a good control of the resolvent of $A$.

In order to formulate the next theorem, let us introduce a  function $\Phi_p : (0,1) \to \mr$ given by
\begin{equation}
  \label{eq:29}
\Phi_p(x)= \frac{\left( W(\frac 1 p e^{\frac 1 p} x) \right)^p}{\left(\frac 1 p -W(\frac 1 p e^{\frac 1 p} x)  \right)^{p+1} x^p}\:,
\end{equation}
where $W : [0,\infty) \to [0,\infty)$ is the Lambert W-function defined by $W(x)e^{W(x)}=x$. This function satisfies the bound
\begin{equation}
  \label{eq:30}
  \Phi_p(x) \leq \frac{(p+1)^{p+1}}{p^p} \cdot\frac{1}{(1-x)^{p+1}}, \quad 0 < x< 1,
\end{equation}
as has been shown in \cite{MR3296588} (see the proof of Corollary 4.3 in that paper).
\begin{thm}\label{thm:4}
Suppose that for some $R \geq r(A)$ and $C_A > 0$ we have 
\begin{equation}\label{resest}
  \|(\lambda-A)^{-1}\| \leq \frac {C_A} {|\lambda|-R}, \quad |\lambda| > R.
\end{equation}
 Then for $s>R$ the following holds:
 \begin{equation}
   \label{eq:31}
  n_{A+K}(s) \leq C_A^p \gamma_p^p \Gamma_p\frac{1}{s^p} \Phi_p\left(\frac{R}{s}\right) \|K\|_{\bi}^p.   
 \end{equation}
Here the right-hand side is bounded from above by
\begin{equation}
  \label{eq:33}
C_A^p \gamma_p^p \Gamma_p \frac{(p+1)^{p+1}}{p^p} \frac{s}{(s-R)^{p+1}} \|K\|_{\bi}^p.
\end{equation}
\end{thm}
\begin{proof}
  For $s>t>R$ let us set $\eps = (t-R)/C_A$ and let us define
$$ \Omega = \{ \lambda: |\lambda|> t\} \quad \text{and} \quad \Omega'=\{\lambda : |\lambda| > s\}.$$
Then $\Omega$ is a simply connected subset of $\hat \mc$ (with $\infty \in \Omega$) and $\Omega' \subset \Omega \subset \mc \setminus \sigma_\eps(A)$ (for the last inclusion we used assumption (\ref{resest})). So we are in a position to apply Theorem \ref{thm:3}: a conformal mapping $\phi$ that maps $\Omega$ onto $\md$ (and $\infty$ onto $0$) is given by
$ \phi(\lambda)= t/\lambda$, so 
$$ r_\Omega(\Omega')= \sup_{\lambda \in \Omega'} |\phi(\lambda)| = \frac t s.$$
We thus obtain from Theorem \ref{thm:3} that for $s>t>R$
$$ n_{A+K}(s) \leq \frac{\gamma_p^p \Gamma_p}{ \eps^p \log(\frac 1 {r_\Omega(\Omega')})} \|K\|_{\bi}^p = \frac{C_A^p\gamma_p^p \Gamma_p}{ (t-R)^p \log( \frac s {t})} \|K\|_{\bi}^p.$$ 
All that remains is to maximize the function $f(t)=(t-R)^{p}\log(\frac s t), t \in (R,s)$. This had already been done in \cite{MR3296588} (see the computations preceeding Theorem 4.2 in that paper), where it was shown that 
$$ \max_{t \in (R,s)} f(t) = \frac{s^p}{\Phi_p(\frac{R}{s})}.$$
This shows the validity of (\ref{eq:31}). The validity of (\ref{eq:33}) follows from (\ref{eq:31}) and estimate (\ref{eq:30}). 
\end{proof}
It is easy to see that for $|\lambda| > \|A\|$ one has
$$ \|(\lambda-A)^{-1}\| \leq 1/(|\lambda|-\|A\|).$$
We thus obtain the following corollary of the previous theorem  (for ${\bi}=\mathcal{S}_p^{(a)}(X)$ this had already been obtained in \cite{MR3296588}).
\begin{cor}\label{cor:1}
If $s> \|A\|$, then 
\begin{equation}
  \label{eq:34}
n_{A+K}(s) \leq   \gamma_p^p \Gamma_p \frac{(p+1)^{p+1}}{p^p} \frac{s}{(s-\|A\|)^{p+1}} \|K\|_{\bi}^p.
\end{equation}
\end{cor}
\begin{rem}
If we consider the number of eigenvalues of $K \in \bi$ in $\{ \lambda : |\lambda| >s\}$, then it is an immediate consequence of assumption (A4) that
$$ n_K(s) \leq \gamma_p^p \frac{\|K\|_{\bi}^p}{s^p}, \qquad s>0.$$
We see that estimate (\ref{eq:34}), up to a multiplicative constant, recovers that result (simply set $A=0$ in that inequality). However, we also see that if $A \neq 0$, then we do \emph{not} obtain that for $s \to \|A\|$
\begin{equation}
  \label{eq:44}
 n_{A+K}(s) = O\left( \frac 1 {(s-\|A\|)^{p}} \right),
\end{equation}
but only that
$$ n_{A+K}(s) = O\left( \frac 1 {(s-\|A\|)^{p+1}} \right).$$
At the moment, for $p \geq 1$, it is unknown whether the larger exponent $p+1$ in case $A \neq 0$ is really necessary (i.e. whether there exist $l_p$-ideals $\bi$ and operators $A$ and $K$ such that (\ref{eq:44}) will not be true),  or whether it is just an artefact of our methods of proof. On the other hand, it is known that for $0<p<1$, (\ref{eq:44}) can indeed not be true, even for Hilbert space operators, as follows from the results in  \cite{Hansmann11} and \cite{MR3033958}.  Moreover, there are at least some indications which suggest  that, also for $p\geq 1$, the situation for $A\neq 0$ will really be different from the case $A=0$, see, e.g., \cite[Example 5.2]{MR3296588}. 
\end{rem}

We conclude this section by noting that there is another (slightly different) way to obtain estimates on the discrete spectrum of $A+K$. For instance, if $\sigma(A)=[a,b]$ is a real interval and we could prove a bound of the form
$$ \|K(\lambda-A)^{-1}\|_\bi \leq \frac{C}{\dist(\lambda,[a,b])^\alpha |\lambda-a|^\beta |\lambda-b|^\gamma}, \quad \lambda \in \rho(A),$$
then we could invoke a theorem of Borichev, Golinskii and Kupin \cite{MR2481997} (which deals with the distribution of zeros of holomorphic functions on the unit disk which grow exponentially on the unit circle, with different rates of growth for different points of the circle), to obtain more precise bounds on the number of eigenvalues of $A+K$ in $\mc \setminus [a,b]$. Of course, to obtain a bound of the above form we would need much more information on the operators involved. We will not further pursue this method in the present paper and simply refer to \cite{DemHanauska13} for some results in this direction (for the Hilbert space setting, see also \cite{MR3077277}). 

\section{Eigenvalues of generators of $C_0$-semigroups}

In this final section we are going to apply our abstract results in a slightly different setting, namely to obtain information on the number of discrete eigenvalues of generators of $C_0$-semigroups. In particular, this will show that our results are applicable to unbounded operators as well. As a general reference for this section we refer to the monograph \cite{MR1721989}.

To begin, let us consider a closed, densely defined operator $H_0$ in the Banach space $X$ and let us assume that $H_0$ is the generator of a $C_0$-semigroup $(T_0(t))_{t \geq 0}$. As usual, in the following we will set $e^{tH_0}:=T_0(t)$. 
\begin{rem}
We recall that $(S(t))_{t \geq 0} \subset \bdd(X)$ is a $C_0$-semigroup on $X$ if the following three conditions are satisfied:
\begin{enumerate}
    \item $S(t+s)=S(t)S(s)$ for all $t,s \geq 0$,
    \item $S(0)=I$,
    \item $[0,\infty) \ni t \mapsto S(t)x \in X$ is continuous for all $x \in X$. 
\end{enumerate}
In this case there exists $M \geq 0$ and $\omega \in \mr$ such that $\|S(t)\| \leq M e^{\omega t}, t \geq 0$. The closed and densely defined operator $G$ in $X$ is the generator of $(S(t))_{t \geq 0}$ if and only if 
$$ \dom(G)=\left\{ x \in X : \lim_{h \downarrow 0} \frac{S(h)x-x}{h} \text{ exists}\right\}$$
and $Gx = \lim_{h \downarrow 0} (S(h)x-x)/h$ for $x \in \dom(G)$. The spectrum of $G$ will always be contained in the half-plane $\{ \lambda \in \mc : \Re(\lambda) \leq \omega\}$. Moreover, we note that for $x_0 \in \dom(G)$ the function $u(t):=S(t)x_0$ is the unique (classical) solution of the abstract Cauchy problem
\begin{equation}
  \label{eq:32}
  \left\{
    \begin{array}{cll}
      \frac{d}{dt} u(t) &=& Gu(t), \quad  t > 0 \\
      u(0)&=& x_0.
    \end{array}\right.
\end{equation}
For an extensive list of examples of $C_0$-semigroups and their generators we refer to \cite{MR1721989} and references therein.
\end{rem}
In the following, let us also assume that $(e^{tH_0})_{t \geq 0}$ is a \emph{contraction} semigroup, meaning that $\|e^{tH_0}\| \leq 1$ for all $t \geq 0$. In particular, this implies that the spectrum of $H_0$ is contained in the left half-plane, i.e. 
$$\sigma(H_0) \subset \{ \lambda \in \mc : \Re(\lambda) \leq 0 \}.$$

Next, let us introduce another closed, densely defined operator $H$ in $X$, which we consider as a small perturbation of $H_0$. More precisely, we assume that $H$ is the generator of a $C_0$-semigroup $(e^{tH})_{t \geq 0}$ on $X$ as well (but which no longer needs to be a contraction) and that for some $a>0$ the semigroup difference 
$e^{aH}-e^{aH_0}$ is an element of some fixed $l_p$-ideal ${\bi}$, $p>0$. 

Under these assumptions the spectrum of $H$ will be contained in a half-plane $\{ \lambda \in \mc : \Re(\lambda) \leq \omega_H\}$, with some $\omega_H$ which might be larger than $0$, and the spectrum of $H$ in the strip $\{ \lambda \in \mc : 0 < \Re(\lambda) \leq \omega_H\}$ will be purely discrete (see the next lemma). Our goal is to obtain a bound on the  number of discrete eigenvalues of $H$ in this strip.
\begin{lemma}\label{lem:inclusion}
Let $\lambda \in \sigma(H)$ with $\Re(\lambda)>0$. Then $\lambda \in \sigma_d(H)$ and $e^{a\lambda} \in \sigma_d(e^{aH})$. Moreover, concerning the algebraic multiplicities of these eigenvalues we have
$$ m(\lambda; H) \leq m(e^{a \lambda}; e^{aH}).$$
\end{lemma} 
\begin{proof}
  From the spectral inclusion theorem (see \cite[Theorem IV.3.6]{MR1721989}) we know that $e^{a \lambda} \in \sigma(e^{aH})$. By assumption, the semigroup difference $e^{aH}-e^{aH_0}$ is compact, so by Weyl's theorem
$$\sigma_{ess}(e^{aH})=\sigma_{ess}(e^{aH_0}) \subset \{ z : |z| \leq \|e^{aH_0}\|\},$$
showing that the set $\Omega:=\{ z : |z| > \|e^{aH_0}\|\}$ is a subset of the unbounded component of $\mc \setminus \sigma_{ess}(e^{aH})$. In particular, the only spectral points of $e^{aH}$ in $\Omega$ are discrete eigenvalues. Since
$$|e^{a\lambda}| = e^{a\Re(\lambda)}>1 \geq \|e^{aH_0}\|,$$
we thus obtain that $e^{a\lambda} \in \sigma_d(e^{aH})$. But then we can apply \cite[Theorem IV.3.6]{MR1721989} again to conclude that $\lambda \in \sigma_d(H)$ and that $m(\lambda; H) \leq m(e^{a\lambda}; e^{aH})$.
\end{proof}

We can now apply Corollary \ref{cor:1} to obtain the main result of this section.
\begin{thm}
  Let $H_0$ and $H$ be generators of a $C_0$-contraction-semigroup and a $C_0$-semigroup on $X$, respectively, and assume that for some $a>0$ the difference $e^{aH}-e^{aH_0}$ belongs to some $l_p$-ideal ${\bi}$, $p>0$. Then for every $s>0$
$$ \mathcal{N}_H( \{ \lambda: \Re(\lambda)>s\}) \leq C_p  \frac{e^{as}}{(e^{as}-1)^{p+1}} \|e^{aH}-e^{aH_0}\|_{\bi}^p,$$
where
$$ C_p=\frac{(p+1)^{p+1} \gamma_p^p \Gamma_p}{p^p}.$$
Here $\gamma_p$ denotes the eigenvalue constant of ${\bi}$ and $\Gamma_p$ is as given in (\ref{eq:1}).
\end{thm}
\begin{proof}
We apply Corollary \ref{cor:1} with $A=e^{aH_0}$ and $K=e^{aH}-e^{aH_0}$ to obtain that for $r> 1 \geq \|A\|$ the following holds:
$$ \mathcal{N}_{A+K}( \{ z : |z| > r \} \leq \frac{C_pr}{(r-\|A\|)^{p+1}} \|K\|_{\bi}^p\leq  \frac{C_pr}{(r-1)^{p+1}} \|K\|_{\bi}^p.$$
Now we choose $r=e^{as}$ and use the fact that from Lemma \ref{lem:inclusion} we obtain that
$$ \mathcal{N}_H( \{ \lambda: \Re(\lambda)>s\}) \leq \mathcal{N}_{A+K}( \{ e^{a\lambda}: \Re(\lambda)>s\}) \leq \mathcal{N}_{A+K}( \{ z : |z| > e^{as} \}).$$
\end{proof} 
In many applications both semigroups $e^{tH_0}$ and $e^{tH}$ (hence also their difference $D_t:=e^{tH}-e^{tH_0}$) are integral operators on some $L_q$-space. In this case there exists a huge variety of results which show that $D_t$ is an element of one of the classical $l_p$-ideals mentioned in Section \ref{sec:examples} if its kernel $D_t(x,y)$ satisfies suitable integrability or smoothness assumptions. We refer to \cite{MR889455} and \cite{MR917067} for a compilation of some of those results. Here, we will restrict ourselves to one easy example.
\begin{ex}[Hille-Tamarkin kernels] \label{ex:hille}
Let $(\Omega,\mu)$ be a $\sigma$-finite measure space and let $1<q<\infty$. Suppose that $H_0$ and $H$ are generators of a $C_0$-contraction-semigroup and a $C_0$-semigroup on $L_q(\Omega,\mu)$, respectively, and assume that for some $a>0$ the operator $D_a:=e^{aH}-e^{aH_0}$ is an integral operator with measurable kernel $d_a:\Omega \times \Omega \to \mc$ satisfying 
\begin{equation}
  \label{eq:36}
  \|d_a\|_{q,q'}:= \left( \int_\Omega \left( \int_\Omega |d_a(x,y)|^{q'} d\mu(y) \right)^{q/{q'}} d\mu(x) \right)^{1/q} < \infty,
\end{equation}
where $1/q+1/{q'}=1$. Then $D_a \in \Pi_q(L_q(\Omega,\mu))$ (the $q$-summing ideal) and $\|D_a\|_{\Pi_q} \leq \|d_a\|_{q,q'}$, see \cite[Theorem 3.a.3]{MR889455} . In particular, taking into account Example \ref{ex:summing} and Remark \ref{rem:persson}, we can use the previous theorem to conclude that for every $s>0$
$$\mathcal{N}_H( \{ \lambda: \Re(\lambda)>s\}) \leq \frac{(p+1)^{p+1}  \Gamma_p}{p^p} \frac{e^{as}}{(e^{as}-1)^{p+1}} \|d_a\|_{q,q'}^p,$$
where $p=\max(2,q)$.   
\end{ex} 

\def\cydot{\leavevmode\raise.4ex\hbox{.}} \def\cprime{$'$}


\begin{thebibliography}{10}

\bibitem{MR2481997}
A.~Borichev, L.~Golinskii, and S.~Kupin.
\newblock A {B}laschke-type condition and its application to complex {J}acobi
  matrices.
\newblock {\em Bull. Lond. Math. Soc.}, 41(1):117--123, 2009.

\bibitem{MR1098497}
B.~Carl and I.~Stephani.
\newblock {\em Entropy, compactness and the approximation of operators},
  volume~98 of {\em Cambridge Tracts in Mathematics}.
\newblock Cambridge University Press, Cambridge, 1990.

\bibitem{b_Davies}
E.~B. Davies.
\newblock {\em Linear operators and their spectra}, volume 106 of {\em
  Cambridge Studies in Advanced Mathematics}.
\newblock Cambridge University Press, Cambridge, 2007.

\bibitem{DemHanauska13}
M.~Demuth and F.~Hanauska.
\newblock On the distribution of the discrete spectrum of nuclearly perturbed
  operators in {B}anach spaces.
\newblock {\em Indian J. Pure Appl. Math.}, 46(4):441-462, 2015.  

\bibitem{MR3296588}
M.~Demuth, F.~Hanauska, M.~Hansmann, and G.~Katriel.
\newblock Estimating the number of eigenvalues of linear operators on {B}anach
  spaces.
\newblock {\em J. Funct. Anal.}, 268(4):1032--1052, 2015.

\bibitem{MR3077277}
M.~Demuth, M.~Hansmann, and G.~Katriel.
\newblock Eigenvalues of non-selfadjoint operators: a comparison of two
  approaches.
\newblock In {\em Mathematical physics, spectral theory and stochastic
  analysis}, volume 232 of {\em Oper. Theory Adv. Appl.}, pages 107--163.
  Birkh\"auser/Springer Basel AG, Basel, 2013.

\bibitem{MR1009163}
N.~Dunford and J.~T. Schwartz.
\newblock {\em Linear operators. {P}art {II}}.
\newblock Wiley Classics Library. John Wiley \& Sons, Inc., New York, 1988.
\newblock Spectral theory. Selfadjoint operators in Hilbert space, With the
  assistance of William G. Bade and Robert G. Bartle, Reprint of the 1963
  original, A Wiley-Interscience Publication.

\bibitem{MR1721989}
K.-J. Engel and R.~Nagel.
\newblock {\em One-parameter semigroups for linear evolution equations}, volume
  194 of {\em Graduate Texts in Mathematics}.
\newblock Springer-Verlag, New York, 2000.
\newblock With contributions by S. Brendle, M. Campiti, T. Hahn, G. Metafune,
  G. Nickel, D. Pallara, C. Perazzoli, A. Rhandi, S. Romanelli and R.
  Schnaubelt.

\bibitem{MR2201310}
F.~Gesztesy, Y.~Latushkin, M.~Mitrea, and M.~Zinchenko.
\newblock Nonselfadjoint operators, infinite determinants, and some
  applications.
\newblock {\em Russ. J. Math. Phys.}, 12(4):443--471, 2005.

\bibitem{MR2391269}
M.~I. Gil'.
\newblock Upper and lower bounds for regularized determinants.
\newblock {\em JIPAM. J. Inequal. Pure Appl. Math.}, 9(1):Article 2, 6, 2008.
 
\bibitem{MR3035142}
M.~I. Gil'.
\newblock Ideals of compact operators with the {O}rlicz norms.
\newblock {\em Ann. Mat. Pura Appl. (4)}, 192(2):317--327, 2013.

\bibitem{MR1130394}
I.~Gohberg, S.~Goldberg, and M.~A. Kaashoek.
\newblock {\em Classes of linear operators. {V}ol. {I}}, volume~49 of {\em
  Operator Theory: Advances and Applications}.
\newblock Birkh\"auser Verlag, Basel, 1990.

\bibitem{MR1744872}
I.~Gohberg, S.~Goldberg, and N.~Krupnik.
\newblock {\em Traces and determinants of linear operators}, volume 116 of {\em
  Operator Theory: Advances and Applications}.
\newblock Birkh\"auser Verlag, Basel, 2000.

\bibitem{b_Gohberg69}
I.~Gohberg and M.~G. Krein.
\newblock {\em Introduction to the theory of linear nonselfadjoint operators}.
\newblock American Mathematical Society, Providence, R.I., 1969.

\bibitem{Hansmann11}
M.~Hansmann.
\newblock An eigenvalue estimate and its application to non-selfadjoint
  {J}acobi and {S}chr\"odinger operators.
\newblock {\em Lett. Math. Phys.}, 98(1):79--95, 2011.

\bibitem{MR3033958}
M.~Hansmann and G.~Katriel.
\newblock From spectral theory to bounds on zeros of holomorphic functions.
\newblock {\em Bull. Lond. Math. Soc.}, 45(1):103--110, 2013.

\bibitem{MR0417827}
J.~S. Howland.
\newblock Analyticity of determinants of operators on a {B}anach space.
\newblock {\em Proc. Amer. Math. Soc.}, 28:177--180, 1971.

\bibitem{MR1840556}
B.~Jacob.
\newblock Zeros of {F}redholm operator valued {$H^p$}-functions.
\newblock {\em Math. Nachr.}, 227:81--97, 2001.

\bibitem{kato}
T.~Kato.
\newblock {\em Perturbation theory for linear operators}.
\newblock Classics in Mathematics. Springer-Verlag, Berlin, 1995.
\newblock Reprint of the 1980 edition.

\bibitem{MR0482266}
H.~K{\"o}nig.
\newblock Interpolation of operator ideals with an application to eigenvalue
  distribution problems.
\newblock {\em Math. Ann.}, 233(1):35--48, 1978.

\bibitem{MR568991}
H.~K{\"o}nig.
\newblock A {F}redholm determinant theory for {$p$}-summing maps in {B}anach
  spaces.
\newblock {\em Math. Ann.}, 247(3):255--274, 1980.

\bibitem{MR889455}
H.~K{\"o}nig.
\newblock {\em Eigenvalue distribution of compact operators}, volume~16 of {\em
  Operator Theory: Advances and Applications}.
\newblock Birkh\"auser Verlag, Basel, 1986.

\bibitem{MR1863710}
H.~K{\"o}nig.
\newblock Eigenvalues of operators and applications.
\newblock In {\em Handbook of the geometry of {B}anach spaces, {V}ol. {I}},
  pages 941--974. North-Holland, Amsterdam, 2001.

\bibitem{MR0248498}
G.~K{\"o}the.
\newblock {\em Topological vector spaces. {I}}.
\newblock Translated from the German by D. J. H. Garling. Die Grundlehren der
  mathematischen Wissenschaften, Band 159. Springer-Verlag New York Inc., New
  York, 1969.

\bibitem{MR0247504}
A.~Persson.
\newblock On some properties of {$p$}-nuclear and {$p$}-integral operators.
\newblock {\em Studia Math.}, 33:213--222, 1969.

\bibitem{MR0243323}
A.~Persson and A.~Pietsch.
\newblock {$p$}-nukleare und {$p$}-integrale {A}bbildungen in {B}anachr\"aumen.
\newblock {\em Studia Math.}, 33:19--62, 1969.

\bibitem{MR582655}
A.~Pietsch.
\newblock {\em Operator ideals}, volume~20 of {\em North-Holland Mathematical
  Library}.
\newblock North-Holland Publishing Co., Amsterdam-New York, 1980.
\newblock Translated from German by the author.

\bibitem{MR917067}
A.~Pietsch.
\newblock {\em Eigenvalues and {$s$}-numbers}, volume~43 of {\em Mathematik und
  ihre Anwendungen in Physik und Technik [Mathematics and its Applications in
  Physics and Technology]}.
\newblock Akademische Verlagsgesellschaft Geest \& Portig K.-G., Leipzig, 1987.

\bibitem{MR2300779}
A.~Pietsch.
\newblock {\em History of {B}anach spaces and linear operators}.
\newblock Birkh\"auser Boston, Inc., Boston, MA, 2007.

\bibitem{MR792906}
F.~Reuter.
\newblock Determinants and a unified approach to estimating resolvents of
  operators in operator ideals of {R}iesz type {$l_p$}.
\newblock {\em Integral Equations Operator Theory}, 8(3):385--401, 1985.
      
\bibitem{MR924157}
W.~Rudin.
\newblock {\em Real and complex analysis}.
\newblock McGraw-Hill Book Co., New York, third edition, 1987.

\bibitem{MR0482328}
B.~Simon.
\newblock Notes on infinite determinants of {H}ilbert space operators.
\newblock {\em Advances in Math.}, 24(3):244--273, 1977.

\bibitem{b_Simon05}
B.~Simon.
\newblock {\em Trace ideals and their applications}.
\newblock American Mathematical Society, Providence, RI, second edition, 2005.

\bibitem{MR2155029}
L.~N. Trefethen and M.~Embree.
\newblock {\em Spectra and pseudospectra}.
\newblock Princeton University Press, Princeton, NJ, 2005.
\newblock The behavior of nonnormal matrices and operators.

\bibitem{MR1180965}
D.~R. Yafaev.
\newblock {\em Mathematical scattering theory}, volume 105 of {\em Translations
  of Mathematical Monographs}.
\newblock American Mathematical Society, Providence, RI, 1992.
\newblock General theory, Translated from the Russian by J. R. Schulenberger.

\end{thebibliography}

\end{document}